\documentclass[12pt]{article}

\usepackage[margin=1in]{geometry}  
\usepackage{graphicx}              
\usepackage{amsmath}               
\usepackage{amsfonts}              
\usepackage{amsthm}                
\usepackage{mathtools}
\usepackage{thmtools}
\usepackage{thm-restate}

\usepackage{enumitem} 
\usepackage[all]{xy}
\usepackage{ amssymb }

\newtheorem{theorem}{Theorem}
\newtheorem{conjecture}[theorem]{Conjecture}

\newtheorem{lemma}[theorem]{Lemma}

\newtheorem{definition}[theorem]{Definition}

\DeclareMathOperator{\vol}{Vol}
\DeclareMathOperator{\argmin}{argmin}

\numberwithin{equation}{section}
\numberwithin{theorem}{section}
\numberwithin{lemma}{section}
\numberwithin{conjecture}{section}
\numberwithin{definition}{section}
\numberwithin{corollary}{section}
\numberwithin{figure}{section}
\numberwithin{table}{section}

\newcommand{\ds}[0]{\displaystyle}

\title{Maximal determinants of sparse zero-one matrices}
\author{
        Daniel Scheinerman \\
        Department of Mathematics\\
        Rutgers University \\
        Piscataway, NJ \\
}
\date{\today}

\begin{document}

\maketitle

\begin{abstract}
We give upper bounds for the determinant of an $n\times n$ zero-one matrix containing $kn$ ones for integral $k$. Our results improve upon a result of Ryser for $k=o(n^{1/3})$. For fixed $k\ge 3$ it was an open question whether Hadamard's inequality could be exponentially improved. We answer this in the affirmative. Our results stem from studying matrices with row sums $k$ and bounding their Gram determinants. Our technique allows us to give upper bounds when these matrices are perturbed.
\end{abstract}

\section{Introduction}\label{sec:introduction}

miosthieioshts

We consider the combinatorial class of $n\times n$ zero-one matrices containing $kn$ ones. We are interested in giving an upper bound on their maximal determinant. We will do this by studying matrices with equal row sums.

\begin{definition}
 Let $R(n,k)$ be the set of  $n\times n$ zero-one matrices whose rows sum to $k$.
\end{definition}

\begin{definition}\label{def:maxdet}
 Let $M_R(n,k)=\max_{A\in R(n,k)} \det(A)$ be the maximum determinant over matrices in $R(n,k)$. 
\end{definition}

The easiest upper bound for $M_R(n,k)$ comes from Hadamard's inequality~\cite{hadamard} which gives $k^{n/2}$ since each row has norm exactly $\sqrt k$. More generally, one can consider matrices containing $kn$ ones.

\begin{definition}
 Let $T(n,k)$ be the set of  $n\times n$ zero-one matrices that contain a total of $kn$ ones.
\end{definition}

\begin{definition}
 Let $M_T(n,k)=\max_{A\in T(n,k)} \det(A)$ be the maximum determinant over matrices in $T(n,k)$. 
\end{definition}

Clearly, $R(n,k)\subset T(n,k)$ and thus $M_R(n,k)\le M_T(n,k)$. Note that if $A\in T(n,k)$ then its rows have average sum $k$ and so using the AM-GM inequality the bound $\det(A)\le k^{n/2}$ still applies. Ryser~\cite{ryser} proved a strengthening of this result.

\begin{theorem}[Ryser's Theorem]\label{thm:ryser}
Let $A$ be an  $n\times n$ zero-one matrix with a total of $t$ ones. Let $k=t/n$ and $\lambda=k(k-1)/(n-1)$. Then
\begin{equation*}
 \det(A) \le k(k-\lambda)^{\frac12(n-1)}
\end{equation*}
with equality holding if and only if $A$ is the incidence matrix of an $(n,k,\lambda)$-design. $\blacksquare$
\end{theorem}

Note that if, for example, $k=\Theta(n)$ then $\lambda = \Theta(n)$ and Theorem~\ref{thm:ryser} gives a large improvement upon Hadamard's inequality. However, if, for example, $k$ is fixed then $\lambda$ is tending to zero and this gives a more modest improvement. We note that if $k \le \sqrt{n}$ then $\lambda<1$ and so $\lambda$ is not an integer. Therefore, we may hope to improve Theorem~\ref{thm:ryser} for matrices that are sufficiently sparse. Our main result is that for $k=o(n^{1/3})$ we can improve the bound given in Theorem~\ref{thm:ryser}. We show that there exists $c_k<\sqrt{k}$ depending only on $k$ such that $M_R(n,k)\le c_k^n$. Moreover, for integral $k$ the bound $M_T(n,k)\le c_k^n$ holds. Thus for $k$ fixed we give an exponential improvement to the bound given by Hadamard's inequality. The existence of such a $c_k<\sqrt{k}$ was only known for $k=2$~\cite{bruhn}. More on this in Section~\ref{sec:k=2}.

Next, we generalize the notions $R(n,k)$ and $M_R(n,k)$ to non-square matrices.

\begin{definition}
 Let $R(m,n,k)$ be the set of $m\times n$ zero-one matrices whose rows sum to $k$.
\end{definition}

\begin{definition}
 For any $m\times n$ real matrix, $A$, where $m\le n$, let $\vol(A) = \sqrt{\det(A A^T)}.$ 
\end{definition}
The matrix $AA^T$ is called the Gram matrix of $A$ and the quantity $\det(A A^T)$ is known as the Gram determinant. See for example~\cite{matrixanalysis}. If $m=n$ we of course have $\vol(A) = |\det(A)|$. For any $m\times n$ real matrix, $A$, with $m\le n$, $\vol(A)$ is the volume of the parallelepiped formed by the rows of $A$. Gram's inequality tells us that $\vol(A)\ge 0$ with equality if and only if the rows of $A$ are linearly dependent in which case we consider the parallelepiped to be degenerate which is consistent with zero volume.

\begin{definition}\label{def:maxvol}
 Let $M_R(m,n,k)=\max_{A\in R(m,n,k)} \vol(A)$. 
\end{definition}

We will repeatedly use the following generalization of Hadamard's inequality. Let $A$ be an $m\times n$ real matrix. If $A$ is partitioned into two horizontal parts $A_1$ and $A_2$ with dimensions $m_1\times n$ and $m_2\times n$ respectively (thus $m_1+m_2=m$) then we have the inequality 
\begin{equation}\label{eqn:vol}
\vol(A) \le \vol(A_1) \vol(A_2). 
\end{equation}
This follows, for example, by Fischer's inequality applied to  the Gram matrix
\begin{equation*}
A A^T = \left(\begin{matrix} A_1 A_1^T & A_1 A^T_2 \\ A_2 A_1^T & A_2 A_2^T \end{matrix}\right).
\end{equation*}

In developing bounds for $M_R(n,k)$ we show more general bounds for $M_R(m,n,k)$. Our basic approach stems from the following. If $M\in R(n,k)$ then it contains $nk$ ones and therefore the columns have average $k$. Thus there exists a collection of at least $k$ rows that share a column of ones. It can be shown that the presence of a column of ones amongst these rows implies that the volume of the parallelepiped spanned by those rows is noticeably smaller than what is implied by Hadamard's inequality. We bound this volume and consider the remaining rows. Since the row sums are constant the remaining rows form a matrix in $R(n-k,n,k)$. We can compute the column averages and iterate this process to give an improved bound.

This paper is organized as follows. In Section~\ref{sec:k=2}, we give background on the special case $k=2$ where $M_R(n,k)$ is known up to a constant factor and is exponentially smaller than $2^{n/2}$. We also give lower bounds for $M_R(n,k)$. In Section~\ref{sec:tworows}, we give an upper bound for $M_R(n,k)$ given by taking the rows in pairs. In Section~\ref{sec:qrows}, we improve this bound by taking the rows in sets of size $q\le k$. In Section~\ref{sec:greedy}, we give, for small $k$, our best bound for $M_R(n,k)$ by greedily selecting the rows for removal. In Section~\ref{sec:geometry_ryser}, we establish some determinant inequalities we will need repeatedly. We use these to prove a generalization of Ryser's theorem for matrices in $R(m,n,k)$. We also give a counterexample to a conjecture of Li, Lin and Rodman~\cite{linesums}. In Section~\ref{sec:knones}, we show that the bound found in Section~\ref{sec:tworows} applies to $M_T(n,k)$ for integral $k$ thus answering a question of Bruhn and Rautenbach~\cite{bruhn}. In Section~\ref{sec:perturbations}, we show that these techniques give upper bounds for perturbations of matrices in $R(n,k)$. We conclude with some open questions.

\section{Special case $k=2$ and lower bounds for $M_R(n,k)$}\label{sec:k=2}

In keeping with the notation of ~\cite{fallat, linesums} we define the following.

\begin{definition}
 Let $S(n,k)$ be the set of  $n\times n$  zero-one matrices whose rows and columns sum to $k$.
\end{definition}

\begin{definition}
 Let $M(n,k)=\max_{A\in S(n,k)} \det(A)$ be the maximum determinant over matrices in $S(n,k)$. 
\end{definition}

Since $S(n,k)\subset R(n,k)$ we of course have $M(n,k) \le M_R(n,k)$. In~\cite{bruhn} the authors study zero-one matrices with at most $2n$ ones. They show the following.

\begin{theorem}\label{thm:bruhn}
If $A$ is an  $n\times n$ zero-one matrix, and each row of $A$ contains at most two ones then $|\det(A)|\le 2^{n/3}$. 
\end{theorem}
Thus, in particular $M_R(n,2)\le 2^{n/3}$. This gives an exponential improvement to the bound given by Theorem~\ref{thm:ryser}. This can be seen to be tight up to a constant factor from the following result found in~\cite{fallat}.

\begin{theorem}\label{thm:fallat}
 $M(4,2)=2$. For $n\ne 4$, if $n=3\ell$ or $n=3\ell+2$ then $M(n,2) = 2^\ell$. If $n=3\ell+1$ then $M(n,2)  = 2^{\ell-1}$. 
\end{theorem}

Furthermore, the following bound for $M_T(n,2)$ is found in~\cite{bruhn}.

\begin{theorem}
 $M_T(n,2) \le 6^{n/6} \approx 1.348^n$. 
\end{theorem}

The authors ask if a similar result holds for matrices with $3n$ ones. We answer this question in the affirmative in Section~\ref{sec:knones}.

One thing the case $k=2$ illuminates is the fact that we do not in general have $M(n,k)=M_R(n,k)$. From Theorem~\ref{thm:fallat} we see that $M(7,2)=2$. However,  $M_R(7,2)=4$. For example, if
\begin{equation*}
 A = 
\left( \begin{matrix}
1  &  1  &  0  &  0  &  0  &  0  &  0   \\
0  &  1  &  1  &  0  &  0  &  0  &  0   \\
1  &  0  &  1  &  0  &  0  &  0  &  0   \\
1  &  0  &  0  &  1  &  0  &  0  &  0   \\
0  &  0  &  0  &  0  &  1  &  1  &  0   \\
0  &  0  &  0  &  0  &  1  &  0  &  1   \\
0  &  0  &  0  &  0  &  0  &  1  &  1   \\
\end{matrix} \right)
\end{equation*}
then $\det(A)=4$. Notice that the rows of $A$ do indeed sum to $2$ however not all columns have sum $2$. So we pose the following question. For which values of $n, k$ is $M(n,k)=M_R(n,k)$? We know from Theorem~\ref{thm:ryser} that equality holds when $\lambda = k(k-1)/(n-1)$ and there is an $(n,k,\lambda)$ combinatorial design.

Next we discuss lower bounds for $M_R(n,k)$. The basic facts below can all be found in~\cite{designtheory}. Let $p$ be a prime. Let $k=p+1$ and $n=p^2+p+1$. Then there exists a projective plane of order $n$. The incidence matrix, $A$, of this projective plane is $n\times n$ with row (and column) sums of $k$. This is a case where Ryser's theorem is tight. Thus $\lambda = 1$ and $\det(A)=M_R(n,k)=k (k-1)^{(n-1)/2}$. Now for any positive integer $t$ let $N=tn$ and form $A^{(t)}$ as the block diagonal matrix with $t$ copies of $A$ along the diagonal. Then $A^{(t)}\in M_R(N,k)$ and has
\begin{align*}
\det(A^{(t)}) &=\det(A)^t \\
 &= k^t (k-1)^{t(n-1)/2} \\
&= k^{N/n} (k-1)^{(N-N/n)/2} \\
&= \left(k^\frac{1}{k^2-k+1} (k-1)^{\frac12 - \frac{1}{2(k^2-k+1)}}\right)^N.
\end{align*}
Thus if $k=p+1$ for $p$ a prime then 
\begin{equation*}
 \limsup_{n\to\infty} M_R(n,k)^{1/n} \ge k^\frac{1}{k^2-k+1} (k-1)^{\frac12 - \frac{1}{2(k^2-k+1)}}.
\end{equation*}
Consequently we cannot hope to find a general upper bound for $M_R(n,k)$ of the form $c_k^{n/2}$ with $c_k<k-1$. For example, if $k=3$ then the construction via the Fano plane gives $\limsup_{n\to\infty} M(n,3)^{1/n} \ge 24^{1/7}\approx1.5746$. One can of course extend this analysis by considering more general combinatorial designs. For example, if $n=11$ and $k=5$ there exists a combinatorial design with $\lambda=2$. In this case the incidence matrix, $A$, has $\det(A)=1215$ and thus $\limsup_{n\to\infty} M_R(n,k)^{1/n} \ge 1215^{1/11} \approx 1.9073$.

\section{Taking rows in pairs}\label{sec:tworows}

The goal of this section is to prove the following theorem.
\begin{theorem}\label{thm:tworows}
For all positive integers $m\le n$ and $k\le n$,
\begin{equation*}
 M_R(m,n,k) \le \left(\sqrt{k^2-1}\right)^{ \frac{m}2 - \frac{n}{2k}} k^{\frac{n}{2k}}.
\end{equation*}
If $m=n$ let  $c_k = \left(\sqrt{k^2-1}\right)^{\frac12\left(1 -\frac{1}{k}\right)} k^{\frac{1}{2k}}$. Then $M_R(n,k) \le c_k^n$. Note that $c_k<\sqrt{k}$.
\end{theorem}

Suppose that $A\in R(m,n,k)$ and there are two rows $r_i$ and $r_j$ that overlap in $a$ ones, i.e. $\langle r_i, r_j\rangle=a$ where $\langle\cdot,\cdot\rangle$ is the dot product.  Then if we let $A_1$ be the $2\times n$ matrix formed by these rows we have 
\begin{equation*}
 A_1 A_1^T =  \left( \begin{matrix}   k & a \\   a & k  \end{matrix} \right )
\end{equation*}
and thus
\begin{equation*}
\vol(A) = \sqrt{k^2-a^2} \le \sqrt{k^2-1}.
\end{equation*}
 which improves on just using Hadamard's inequality for these rows. Hadamard's inequality tells us that $M_R(m,n,k)\le k^{m/2}$. We now use these ideas to show Theorem~\ref{thm:tworows}.

\begin{proof}[Proof of Theorem~\ref{thm:tworows}]
Any $A\in R(m,n,k)$ contains $mk$ ones. If $mk>n$ then by the pigeon hole principle there is a column with at least two ones. Thus there exist rows $r$ and $s$ such that $\langle r,s\rangle \ge 1$. Let $M_1$ be the $2\times n$ matrix consisting of rows $r$ and $s$ and $A_2$ be the matrix consisting of the remaining $m-2$ rows. Then $\vol(A_1) \le \sqrt{k^2-1}$. Note that $A_2 \in R(m-2,n,k)$ and thus by equation~\eqref{eqn:vol}, $M_R(m,n,k) \le \sqrt{k^2-1} M_R(m-2,n,k)$. Iterating this procedure $t$ times we have \begin{equation*}
M_R(m,n,k) \le \left(\sqrt{k^2-1}\right)^t M_R(m-2t,n,k)
\end{equation*}
 with the process halting once $(m-2t)k \le n$. Thus $m-2t \le n/k$. So $M_R(m-2t,n,k) \le k^{\frac{n}{2k}}$ by Hadamard's inequality. Further $\ds t\ge \frac m2-\frac{n}{2k}$ so we obtain
\begin{equation*}
M_R(m,n,k) \le \left(\sqrt{k^2-1}\right)^{\frac m2-\frac{n}{2k}} \sqrt{k}^{n/k}
\end{equation*}
as desired. Substituting $m=n$ gives the bound for $M_R(n,k)$.
\end{proof}

Theorem~\ref{thm:tworows} gives a better bound for $M(n,k)$ than Theorem~\ref{thm:ryser} provided $k$ is small. This is summarized in Theorem~\ref{thm:beat_ryser}. 

\begin{restatable}{theorem}{beatryser}\label{thm:beat_ryser}
Let $c_k$ be defined as in Theorem~\ref{thm:tworows} and $\lambda=k(k-1)/(n-1)$ as in Theorem~\ref{thm:ryser}. If $k=o(n^{1/3})$ then for $n$ sufficiently large, $c_k^n < k (k-\lambda)^{(n-1)/2}$. 
\end{restatable}

The proof of Theorem~\ref{thm:beat_ryser} is straightforward, but tedious. It can be found in the appendix. We just sketch the heuristics here. The growth of $c_k^n$ is, roughly, $\sqrt{k^2-1}^{n/2}$. Ryser's bound is, roughly, $(k-\lambda)^{n/2}$. Since $\sqrt{k^2-1}<k-\frac{1}{2k}$ the result is achieved provided $k-\frac{1}{2k} < k-\lambda$ and thus $\frac{1}{2k}>\lambda = k(k-1)/n$ which holds when $k=o(n^{1/3})$.

\subsection*{Example, $k=3$}

Let $k=3$ and $n=1000$. We give three bounds.

\begin{enumerate}[nosep]
\item Using Hadamard's inequality $M_R(n,k) \le k^{n/2} = 3^{500} \approx 3.64\times 10^{238}$.
\item Ryser's result has $\lambda =2/333= 0.\overline{006}$ and gives the bound $M(n,k)\le 3 (3-\lambda)^{\frac{1000-1}2}=3 (2.99399\ldots)^{499.5}\approx 2.31\times 10^{238}.$
\item  Theorem~\ref{thm:tworows} gives the bound $c_k^n$ where $c_k \approx 1.6984<\sqrt 3 $ and thus $M_R(n,k)\le c_k^n \approx 1.08\times 10^{230}.$
\end{enumerate}

\section{Taking rows in sets of size $q$}\label{sec:qrows}

In this section we generalize our approach in Section~\ref{sec:tworows} to removing from $M\in R(m,n,k)$ rows in sets of size $q$. If we have $q$ rows that each have a common one coordinate then their Gram matrix will have elements $k$ on the diagonal and elements greater than or equal to one off the diagonal. Thus we have the following definition.

\begin{definition}
Let $S_{n,a,k}$ be the $n\times n$ matrix with diagonal elements equal to $k$ and off-diagonal elements equal to $a$. If $I_n$ is the $n\times n$ identity matrix and $J_n$ is the $n\times n$ all ones matrix we can write $S_{n,a,k} = a J_n + (k-a) I_n$.
\end{definition}

Notice that the incidence matrix of an $(n,k,\lambda)$-design is $S_{n,\lambda,k}$. We will make use of the following lemma which will be proved in Section~\ref{sec:geometry_ryser}.

\begin{restatable}{lemma}{special} \label{lem:special}
We have $\det(S_{n,a,k})=(a(n-1)+k)(k-a)^{n-1}$ and $S_{n,a,k}$ is positive definite if $a<k$. Further, for any positive definite $n\times n$ matrix $A$ such that $A$ has diagonal elements $k$ and $A\ge S_{n,a,k}$ we have $\det(A) \le \det(S_{n,a,k})$.
\end{restatable}

In particular, we will make use of the special case of Lemma~\ref{lem:special} that $\det(S_{q,1,k}) =  (q+k-1)(k-1)^{q-1}$ which has maximal determinant over all $q\times q$ positive definite $q\times q$ matrices with diagonal elements $k$ and non-diagonal elements at least one. This generalizes the trivial fact, used in Section~\ref{sec:tworows}, that if $A=\left(\begin{matrix} k & a \\ a & k \end{matrix}\right)$ with $a\ge 1$ then $\det(A)\le k^2-1$.

\begin{theorem}\label{thm:qrows}
 Let $q$ be an integer with $1\le q \le k$. We have,
\begin{equation*}
M_R(m,n,k) \le \left(\sqrt{ (q+k-1)(k-1)^{q-1}}   \right)^{\frac mq - \frac nk \frac{q-1}q } k^{\frac{n (q-1)}{2k}}.
\end{equation*}
If $m=n$, let  
\begin{equation}\label{eqn:ck_qrows}
 c_{q,k} = (q+k-1)^{\frac 1{2q} \left(1-\frac{q-1}k \right)} (k-1)^{\frac 12 \frac {q-1}q \left(1-\frac{q-1}k \right)} k^{\frac{(q-1)}{2k}}.
\end{equation}
Then $M_R(n,k) \le c_{q,k}^n.$
\end{theorem}

\begin{proof}
Suppose we have $A\in R(m,n,k)$. The number of ones in $A$ is $mk$. The average number of ones in a column is $mk/n$. So if $mk/n>q-1$ then there is some column containing at least $q$ ones. Let $R_q$ be an arbitrary submatrix formed by taking $q$ rows that have a column of ones. Then $R_q R_q^T \ge S_{q,1,k}$ with equality if all other column sums of $R_q$ are $0$ or $1$. Thus Lemma~\ref{lem:special} tells us that $\vol(R_q) \le\sqrt{ (q+k-1)(k-1)^{q-1}}$. We remove these rows and iterate $t$ times. So we have 
\begin{equation*}
M_R(m,n,k) \le \left(\sqrt{ (q+k-1)(k-1)^{q-1}} \right)^t M_R(m-qt,n,k)
\end{equation*}
where $t$ must satisfy $(m-qt)k/n > q-1$. Thus $m-qt>\frac nk (q-1)$ and $t<\frac mq - \frac nk \frac{q-1}q$. Thus we have 
\begin{equation*}
M_R(m,n,k) \le \left(\sqrt{ (q+k-1)(k-1)^{q-1}}   \right)^{\frac mq - \frac nk \frac{q-1}q } k^{\frac{n (q-1)}{2k}}.
\end{equation*}
If we let $m=n$, then we have
\begin{align*}
M_R(m,n,k) &\le \left( (q+k-1)(k-1)^{q-1}   \right)^{\frac n{2q} \left(1-\frac{q-1}k \right)}  k^{\frac{n (q-1)}{2k}} \\
 &= (q+k-1)^{\frac n{2q} \left(1-\frac{q-1}k \right)} (k-1)^{\frac n2 \frac {q-1}q \left(1-\frac{q-1}k \right)} k^{\frac{n (q-1)}{2k}} \\
&=c_{q,k}^n
\end{align*}
with $c_{q,k}$ as defined in equation~\eqref{eqn:ck_qrows}.
\end{proof}

Notice that $c_k$ as defined in Theorem~\ref{thm:tworows} is equivalent to $c_{2,k}$. In Theorem~\ref{thm:point44} in the appendix we show that $c_{q,k}$ is minimized when $q\approx 0.44 k$. For example, when $k=49$, we computed $c_{q,k}$ for $q=1,2,\ldots,k$. In this case $c_{1,k}=\sqrt{k}=7$. To visualize we plotted $q$ versus $\sqrt{k}-c_{q,k}$. The peak of this graph tells us the optimal choice of $q$. See figure~\ref{fig:qrows_k_49_subtracted}. In this case the optimal choice of $q$ is $q_*=\argmin_q c_{q,k}= 23$. In this case $q_*/k \approx 0.47$. We can calculate $c_{23,49}\approx 6.9931$. The plot shows that, in terms of a discrepancy from $\sqrt{k}$, using $q=23$ versus the simpler approach using $q=2$ outlined in Section~\ref{sec:tworows} gives substantial improvement.

\begin{figure}[ht]
\centering
\includegraphics[width=5.5in]{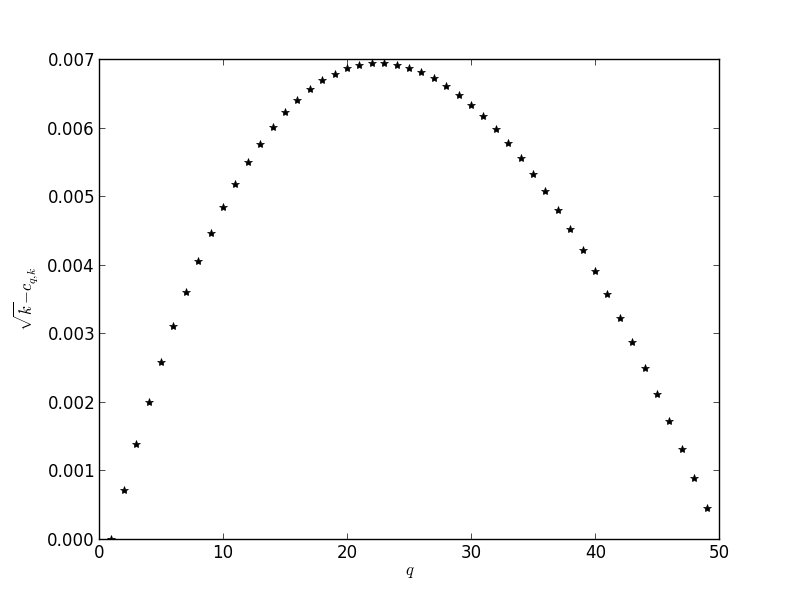}
\caption{$q$ versus $\sqrt{k}-c_{q,k}$ for $k=49$. The peak is at $(23,6.9931)$.}
\label{fig:qrows_k_49_subtracted}
\end{figure}

\subsection*{Example, $k=17$}

From Theorem~\ref{thm:qrows} we have $M_R(n,17)\le c_{q,17}^n$. We give the following progressively better bounds.
\begin{enumerate}[nosep]
 \item Hadamard's inequality is $c_{1,17}=\sqrt{17}\approx 4.1241$.
 \item Using $q=2$ rows at a time we have $c_{2,17}\approx 4.1197$.
 \item For $q\in[17]$, the minimum $c_{q,17}$ occurs when $q=8$. We have $c_{8,17}\approx 4.1111$. 
\end{enumerate}
In Section~\ref{sec:greedy}, we show that we can further improve our bound on $M_R(n,17)$.

\section{Greedily grab rows}\label{sec:greedy}

The main result of this section is Theorem~\ref{thm:greedy_det} below. As in the previous sections we show this by establishing a more general bound for $M_R(m,n,k)$. This is Theorem~\ref{thm:greedy_volume}. For constant $k$, the bound in Theorem~\ref{thm:greedy_det} is asymptotically better than that in Theorem~\ref{thm:tworows} and one can numerically check is better than Theorem~\ref{thm:qrows} for $k\le27$. See Theorem~\ref{thm:greedy_beats_tworows} in the appendix.
\begin{theorem}\label{thm:greedy_det}
 Let 
\begin{equation*}
\alpha_k = \sqrt{\frac{(2k-1)!}{(k-1)!} } (k-1)^{\frac14(k^2-k)}
\end{equation*}
 and 
\begin{equation*}
\beta_k = \left(k + \frac{k}{H_{k}} -1  \right)^{\frac12 (H_k/k)}  (k-1)^{\frac12(1-H_k/k)} 
\end{equation*}
where $H_j=\sum_{i=1}^j 1/i$ is the $j$-th harmonic number. Then $M_R(n,k) \le \alpha_k \beta_k^n.$
\end{theorem}

Suppose we have $A\in R(m,n,k)$. The number of ones in $A$ is $mk$. Thus the column averages are $mk/n$. Thus if we let $r=\lceil mk/n \rceil$ we can find $r$ rows that share a column of ones and thus by Lemma~\ref{lem:special} their volume is at most $\sqrt{\det(S_{r,1,k})} = (r+k-1)^{1/2} (k-1)^{(r-1)/2}$. Recursively, we will then use the bound 

\begin{equation*}
M_R(m,n,k) \le (r+k-1)^{1/2}(k-1)^{(r-1)/2} M_R(m-r,n,k).
\end{equation*}

We will begin by removing $r$ rows but as the number of rows in $A$ diminishes, the number of rows we can remove at each iteration will ultimately diminish to one in which case we are using Hadamard's inequality. For example, if $m=100$, $n=200$ and $k=17$ we will begin by removing $\lceil 100\cdot 17/200 \rceil = 9$ rows. We now have a matrix with $100-9=91$ rows and next we greedily remove $\lceil (100-9)\cdot 17 /200 \rceil = 8$ rows. The sequence of removals, $Q$, in this case is $$Q=(9, 8, 8, 7, 6, 6, 5, 5, 4, 4, 4, 3, 3, 3, 3, 2, 2, 2, 2, 2, 2, 1, 1, 1, 1, 1, 1, 1, 1, 1, 1).$$

Let $a_i$ be the number of times $i$ appears in $Q$. In the above example $a_9=1$ and $a_8=2$. Let $m_r=m$ and for $i<r$ let $m_i$ be the number of rows remaining just prior to removing $a_i$ sets of $i$ rows. Thus $m_0=0$. As above we have 
\begin{equation*} 
r = \left\lceil \frac{mk}n \right\rceil.
\end{equation*}
For $i=1,\ldots,r$ we have
\begin{equation*} 
m_{i-1} = m_i-i a_i
\end{equation*}
For $i\le r$ if we have $m_{i-1}$ rows we just removed $ia_i$ rows. Thus the column average is at most $i-1$. However, if we had $m_{i-1} + i$ rows then the column average must have exceeded $i-1$ as we were able to remove $i$ rows. Thus we have
\begin{equation*}
 \frac{m_{i-1} k}{n} \le i-1 < \frac{(m_{i-1}+i)k}{n}.
\end{equation*}
Rearranging, we have 
\begin{equation}\label{eqn:m_i}
 \frac{(i-1)n}k - i < m_{i-1} \le \frac{(i-1)n}k
\end{equation}
We stress that a similar bound need not hold for $m_r=m$ as this does not arise from just having removed sets of $r+1$ rows. However, we will note momentarily that the bound does hold for $m_r$ when $m=n$. For $2\le i \le r$ we have 
\begin{equation}\label{eqn:a_i}
a_{i-1}=\frac{m_{i-1}-m_{i-2}}{i-1} .
\end{equation}
Subtracting the upper bound for $m_{i-1}$ and the lower bound for $m_{i-2}$ from equation~\eqref{eqn:m_i} and substituting into equation~\eqref{eqn:a_i} gives an upper bound for $a_{i-1}$. Similarly we subtract the lower bound for $m_{i-1}$ and the upper bound for $m_{i-2}$ to get a lower bound for $a_{i-1}$. We obtain
\begin{equation}\label{eqn:a_i_bound}
  \frac{n}{k(i-1)} - \frac{i}{i-1}  < a_{i-1} < \frac{n}{k(i-1)} + 1.
\end{equation}
So we see that for $i<r$, the approximation $a_i\approx \frac{n}{ki}$ is quite good. Finally, we seek a bound for $a_r$. We have 
\begin{align*}
 a_r &= \frac{m-m_{r-1}}{r} \\
     &< \frac{m-\left(\frac{(r-1)n}{k}-r\right)}r \\
     &= \frac{1}{r} \left( m + \frac nk\right) - \frac nk +1 \\
     &\le \frac{1}{mk/n}\left( m + \frac nk\right) - \frac nk +1  \\
     &=  \frac{n^2}{k^2 m} + 1
\end{align*}
We note that if $n | mk$, for example when $n=m$ then this approximation is quite precise since $r=mk/n$. In the case $m=n$, we have $r=k$ and $a_k \le n/k^2+1$ which is consistent with equation~\eqref{eqn:a_i_bound}.

Now that we have bounded $a_i$ for $i=1,\ldots,r$ we can give an upper bound for $M_R(m,n,k)$. We have

\begin{align}
M_R(m,n,k) &\le \prod_{i=1}^r \left(\sqrt{(i+k-1)(k-1)^{i-1}}\right)^{a_i} \nonumber \\
&= \left(\prod_{i=1}^{r-1} \left((i+k-1)(k-1)^{i-1}\right)^{a_i/2}    \right) \left( (r+k-1)  (k-1)^{r-1}\right)^{a_r/2} \nonumber \\
&= \left( \prod_{i=1}^{r-1} (i+k-1)^{a_i/2} \right)  \left( (k-1)^{\frac12\sum_{i=1}^{r-1} (i-1) a_i } \right)  \left( (r+k-1)  (k-1)^{r-1}\right)^{\frac12 a_r} \nonumber  \\
&\le X_{r-1} \cdot Y_{r-1} \cdot  Z_r \label{eqn:XYZgen}
\end{align}
where 
\begin{align}
X_r &= \prod_{i=1}^{r} (i+k-1)^{\frac12\left(\frac{n}{ki}+1\right)} \label{eqn:Xdef} \\
Y_r &= (k-1)^{\frac12\sum_{i=1}^{r} (i-1) a_i } \\
Z_r &= \left((r+k-1)(k-1)^{r-1}\right)^{\frac12 \left(\frac{n^2}{k^2m}+1\right)}
\end{align}
Note that in the case $m=n$, we have $r=k$ and the estimate $a_k\le\frac{n}{k^2}+1$ agrees with the bound $a_i\le \frac{n}{ik}+1$ and thus
\begin{equation}\label{eqn:XYsquare}
 M_R(n,k) \le X_k Y_k.
\end{equation}

We begin by bounding $X_r$.
\begin{align*}
X_r  &= \prod_{i=1}^{r} (i+k-1)^{\frac12\left(\frac{n}{ki}+1\right)} \\
 &= \sqrt{\frac{(r+k-1)!}{(k-1)!}} \left(\prod_{i=1}^{r} (i+k-1)^{1/i}\right)^{\frac{n}{2k}}. \\
\end{align*}
Let $F(r,k) = \prod_{i=1}^{r} (i+k-1)^{1/i}$. Then $\log( F(r,k)) = \sum_{i=1}^{r}\frac{\log(i+k-1)}{i}$. Denote by $H_j=\sum_{i=1}^j 1/i$ the $j$-th harmonic number. Since $\log$ is a concave function we have, using Jensen's inequality,
\begin{align*}
\frac{\sum_{i=1}^{r} \frac{\log(i+k-1)}{i}}{\sum_{i=1}^{r} \frac1i} &\le \log\left(  \frac{\sum_{i=1}^{r} \frac{i+k-1}{i}}{\sum_{i=1}^{r} \frac1i}  \right) \\
&= \log \left(\frac{r + (k-1)H_{r-1}}{H_{r}}  \right) \\
&= \log \left(k + \frac{r}{H_{r}} -1  \right) 
\end{align*} 
and therefore
\begin{equation*}
\log( F(r,k)) \le \log\left(k + \frac{r}{H_{r}} -1  \right) H_{r}.
\end{equation*}
So
\begin{equation*}
F(r,k) \le \left(k + \frac{r}{H_{r}} -1  \right)^{H_{r}}.
\end{equation*}
Finally, we see that
\begin{equation*}
X_r \le \sqrt{\frac{(r+k-1)!}{(k-1)!}}  \left(k + \frac{r}{H_{r}} -1  \right)^{\frac{n H_{r}}{2k}}.
\end{equation*}

Next, we study the second factor in equation~\eqref{eqn:XYZgen}. Let $T_r=\sum_{i=1}^{r} (i-1)\left(\frac{n}{ik}+1\right)$. Then $B_r = (k-1)^{T_r/2}$. We have
\begin{align*}
T_r &= \sum_{i=1}^{r} \frac{n}{k}-1 +i - \frac{n}{ik}\\
   &= r\left(\frac nk - 1\right) + \frac{r(r+1)}{2} - \frac nk H_{r}\\
 &= (r-H_{r})\frac nk + \frac12(r^2-r).
\end{align*}
Thus,
\begin{equation*}
Y_r = (k-1)^{\frac{n}{2k}(r-H_{r})} (k-1)^{\frac14(r^2-r)}.
\end{equation*}

If we substitute our bound for $X_{r-1}$ and $Y_{r-1}$ and $Z_r$ into equation~\eqref{eqn:XYZgen} we obtain the following theorem.

\begin{theorem}\label{thm:greedy_volume}
\begin{align}
\begin{split}
M_R(m,n,k) \le & \sqrt{\frac{(r+k-2)!}{(k-1)!}}    (k-1)^{\frac14(r^2-3r+2)} ~~\times \\
&  \left(k+ \frac{r-1}{H_{r-1}} -1  \right)^{\frac{n H_{r-1}}{2k}}  (k-1)^{\frac{n}{2k}(r-H_{r-1}-1)}\left( (r+k-1)  (k-1)^{r-1}\right)^{\frac12\left(\frac{n^2}{k^2m}+1\right)} \\
\end{split}
\end{align}
where we have arranged the terms that depend on $r$ and $k$ only on the first row and the terms that depend on $n$ and $m$ on the second.
\end{theorem}

If we have a square matrix, $m=n$, then equation~\eqref{eqn:XYsquare} gives us
\begin{align*}
M_R(n,k) &\le X_k Y_k \\
   &\le \sqrt{\frac{(2k-1)!}{(k-1)!}}  \left(k + \frac{k}{H_{k}} -1  \right)^{\frac{n H_{k}}{2k}} (k-1)^{\frac{n}{2k}(k-H_{k})} (k-1)^{\frac14(k^2-k)} \\
&= \sqrt{\frac{(2k-1)!}{(k-1)!} } (k-1)^{\frac14(k^2-k)} \left(\left(k+\frac{k}{H_{k}} -1  \right)^{\frac{H_{k}}{2k}}  (k-1)^{\frac{1}{2k}(k-H_{k})}\right)^n
\end{align*}
establishing Theorem~\ref{thm:greedy_det} above.

\subsection*{Examples, $k=3$ and $k=17$}

For $k=3$ we have the following,
\begin{enumerate}[nosep]
\item In Section~\ref{sec:tworows} we saw $c_{2,3}=1.6984$. So $M_R(n,3)\le 1.6984^n$.
\item  Theorem~\ref{thm:greedy_det} tells us that $\alpha_3\approx 21.91$ and $\beta_3 =(40/11)^{11/36}2^{7/36} \approx 1.6977$ and $M_R(n,3)\le 21.91\times 1.6977^n$. In this case the strategy is, roughly, to use $n/9$ sets of three rows, $n/6$ sets of two rows, and apply Hadamard's inequality to the remaining $n/3$ rows. 
\end{enumerate}

\noindent For $k=17$ we have the following progressively (asyptotically) better bounds. These are visualized in figure~\ref{fig:qrows_k_17_subtracted_with_line}.
\begin{enumerate}[nosep]
\item $M_R(n,17)\le c_{2,17}^n \approx 4.1197^n$. 
\item $M_R(n,17)\le c_{8,17}^n \approx 4.1111^n$. 
\item Using Theorem~\ref{thm:greedy_det} we can compute $\alpha_{17}\approx 4.8887\times 10^{93}$ and $\beta_{17}\approx 4.1104$. Thus $M_R(n,17)\le 4.8887\times10^{93} \cdot 4.1104^n$.
\end{enumerate}


\begin{figure}[ht]
\centering
\includegraphics[width=5.5in]{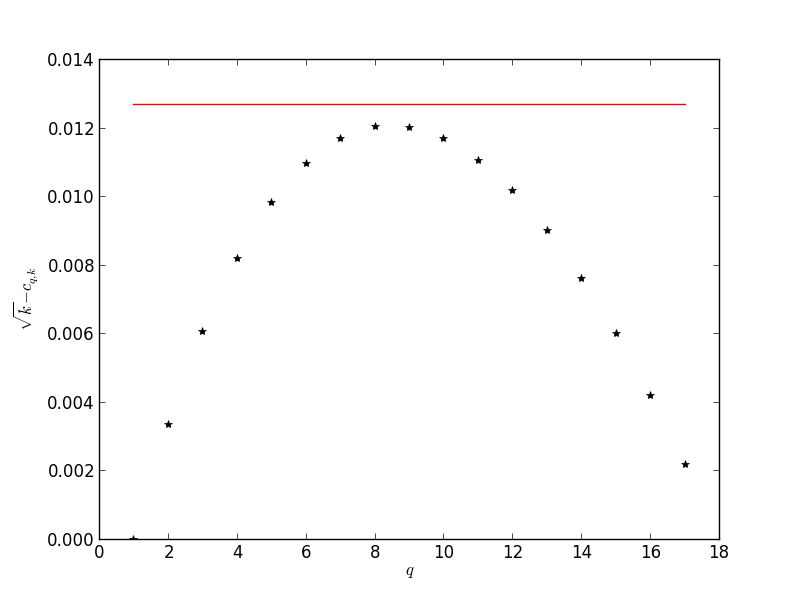}
\caption{$q$ versus $\sqrt{k}-c_{q,k}$ for $k=17$. We draw a red line at height $\sqrt{k}-\beta_k$ to show that, for $k=17$, the greedy approach gives a better bound.}
\label{fig:qrows_k_17_subtracted_with_line}
\end{figure}

We note that for general $k$ our bound for $\alpha_k$ is quite large. Due to the uncertainty of the $a_i$, the product computed in equation~\ref{eqn:Xdef}, multiplies this uncertainty $k$ times. Our goal was to minimize $\beta_k$ and as we were interested in the case where $k$ is constant. However, for any given $n$ we can compute a practical bound. For example, if $k=17$ as above and $n=1000$ then the bound $M_R(1000,17)\le c_{8,17}^{1000}\approx 9.0074\times 10^{613}$. If we were to just use the bound $M_R(1000,17)\le \alpha_{17} \beta_{17}^{1000}$ we would obtain $M_R(1000,17)\le  3.7674\times 10^{707}$ which is a worse bound. However, we can in this case exactly compute the $a_i$. These counts can be found in Table~\ref{tab:greedy_k_17_n_1000}. They give the improved bound $M_R(1000,17)\le 9.3551\times 10^{612}$.

\begin{table}[ht]
\centering
\begin{tabular}{|c|c|}
\hline
$ q  $  &  $  a_q $  \\
\hline
17  &  4  \\
16  &  4  \\
15  &  3  \\
14  &  5  \\
13  &  4  \\
12  &  5  \\
11  &  5  \\
10  &  6  \\
9  &  7  \\
8  &  7  \\
7  &  8  \\
6  &  10  \\
5  &  12  \\
4  &  14  \\
3  &  20  \\
2  &  29  \\
1  &  57  \\
\hline
\end{tabular}
\caption{Counts for greedy row removal for $k=17$ and $n=1000$.}
\label{tab:greedy_k_17_n_1000}
\end{table}

\section{A generalization of Ryser's theorem}\label{sec:geometry_ryser}

In this section we state and establish some facts about the determinants of positive definite matrices. We will use these to prove a generalization of Ryser's theorem for matrices in $R(m,n,k)$. In~\cite{olkin} the author proves the following
\begin{lemma}\label{lem:correlation}
Let $A$ be an $n\times n$, positive definite matrix with diagonal elements $a_{i,i}=1$. Let $\bar{a}=\frac{1}{n(n-1)}\sum_{i\ne j} a_{i,j}$ be the average of the off-diagonal elements. Let $\tilde{A}$ be an $n\times n$ matrix such that $\tilde{a}_{i,i}=1$ and $\tilde{a}_{i,j}=\bar{a}$ for $i\ne j$.  Then $\lambda(\tilde{A})\prec \lambda(A)$ (the eigenvalues of $\tilde{A}$ are majorized by the eigenvalues of $A$) and thus $\det(A)\le \det(\tilde{A})$.
\end{lemma}

Notice, that via rescaling the requirement $a_{i,i}=1$ can be replaced by any constant on the diagonal. Recall that $S_{n,a,k}$ is the $n\times n$ matrix with diagonal elements $k$ and off-diagonal elements $a$. We now restate and prove Lemma~\ref{lem:special}.

\special*

\begin{proof}
To see $\det(S_{n,a,k})=(a(n-1)+k)(k-a)^{n-1}$ we find the eigenvalues. If $u$ is the all ones vector, then $S_{n,a,k} u = (an+k-a)u$ thus $S_{n,a,k}$ has the  eigenvalue $an+k-a = a(n-1)+k$. Further if $v$ is in the codimension one subspace of vectors whose coordinates sum to zero then $S_{n,a,k}v = (k-a) v$ and thus $S_{n,a,k}$ has the eigenvalue $(k-a)$ with multiplicity $n-1$. Thus $\det(S_{n,a,k})=(a(n-1)+k)(k-a)^{n-1}$. If $a<k$ all eigenvalues are positive. 

Next, fix $n,k$ and let $f(x) = \det(S_{n,x,k})= (x(n-1)+k)(k-x)^{n-1}$. Then
\begin{align*}
\frac{d}{dx} f(x) &= (n-1)(k-x)^{n-1} - (x(n-1)+k) (n-1) (k-x)^{n-2} \\
&=(n-1)(k-x)^{n-2} \left[(k-x) - (x(n-1) + k)\right] \\
&= (n-1)(k-x)^{n-2} (-xn) \\
&<0
\end{align*}
for all $x<k$. Thus $f(x)$ is a decreasing function for $x<k$. If $\bar{a}$ is the average of the off-diagonal elements of $A$ then we have $\tilde{A} = S_{n,\bar{a},k}$ and $a\le \bar{a}$. From Lemma~\ref{lem:correlation} we have $\det(A) \le \det(\tilde{A})$.  Since $\det(S_{n,x,k})$ is decreasing we have $\det(\tilde{A}) \le \det(S_{n,a,k})$. Combining these two inequalities gives the result.
\end{proof}

We use the above lemmas to prove the following generalization of Ryser's theorem (Theorem~\ref{thm:ryser}).

\begin{theorem}\label{thm:ryser_gen}
Let $A\in R(m,n,k)$. Let $\mu = \frac{k}{m-1}\left(\frac{mk}n-1\right)$. Then 
\begin{equation}\label{eqn:ryser_gen}
\vol(A) \le  k\sqrt{\frac{m}{n}} (k-\mu)^{\frac{m-1}{2}}.
\end{equation}
\end{theorem}
Notice that if $m=n$ then $\mu = k(k-1)/n = \lambda$ and we recover Theorem~\ref{thm:ryser}.

\begin{proof}
Let $A\in R(m,n,k)$ and consider the Gram matrix, $AA^T$. We have $\vol(A)=\sqrt{\det(AA^T)}$. The diagonal elements of $AA^T$ are all $k$. Let $b_j$ be the number of ones in column $j$ of $A$. We have
\begin{equation*}
\sum_{j=1}^n b_j = mk
\end{equation*}
 If there are $b_j$ ones in column $j$ then the number of ordered pairs of distinct rows $(r,s)$ that overlap in these ones is $2\binom{b_j}2$. So we have
\begin{align*}
\sum_{\substack{r,s\in \text{rows}(A)\\r\ne s}} \langle r,s\rangle &= \sum_{j=1}^n 2\binom{b_j}{2} \\
  &= \sum_{j=1}^n b_j^2 - \sum_{j=1}^n b_j \\
 &=  \sum_{j=1}^n b_j^2 - mk
\end{align*}
The sum of the squares of the $b_j$ is minimized when they are all equal. So we get the lower bound
\begin{equation*}
  \sum_{\substack{r,s\in \text{rows}(A)\\r\ne s}} \langle r,s\rangle \ge n \left(\frac{mk}n\right)^2 - mk =  mk \left(\frac{mk}n-1\right) 
\end{equation*}

The average off-diagonal entry of $AA^T$ can then be bounded.
\begin{equation*}
\frac{1}{m(m-1)} \sum_{\substack{r,s\in \text{rows}(A)\\r\ne s}} \langle r,s\rangle  \ge  \frac{1}{m(m-1)} \left( mk \left(\frac{mk}n-1\right) \right) = \frac{k}{m-1}\left(\frac{mk}n-1\right) = \mu.
\end{equation*}
 Notice that if $m=n$ then $\mu = k(k-1)/(n-1)$ and thus $\mu=\lambda$ as in Theorem~\ref{thm:ryser}. Also, notice that this only gives useful information if $\mu>0$ and thus $m>n/k$. This is not surprising as otherwise $mk<n$ and then we can arrange the rows orthogonally. Thus, Lemma~\ref{lem:special} gives us
\begin{align*}
\det(A) &\le \det(S_{m,\mu,k}) \\
&= (\mu (m-1)+k)(k-\mu)^{m-1} \\
&=k^2 \frac{m}{n} (k-\mu)^{m-1}
\end{align*}
Taking the square root gives equation~\eqref{eqn:ryser_gen}. 
\end{proof}

\subsection{Counterexample to a conjecture of Li, Lin and Rodman}

 Conjecture 4.8 of~\cite{linesums} states that if $\lambda = k(k-1)/(n-1)$ and $A\in S(n,k)$ is non-singular and the off-diagonal entries, $x$, of $AA^T$ and $A^TA$ satisfy $|x-\lambda|<1$ then $|\det(A)|=M(n,k)$. We give the following counterexample. Let $n=10$ and $k=3$. In this case $\lambda = 3\cdot 2 / 9 = 2/3$. First observe that $M(10,3)\ge 48$ since if 
\begin{equation*}
B = \left( \begin{matrix}
0  &  1  &  0  &  0  &  0  &  1  &  0  &  0  &  1  &  0   \\
0  &  0  &  0  &  0  &  0  &  1  &  1  &  1  &  0  &  0   \\
1  &  0  &  0  &  0  &  0  &  0  &  1  &  0  &  1  &  0   \\
0  &  1  &  0  &  0  &  1  &  0  &  0  &  1  &  0  &  0   \\
0  &  0  &  0  &  1  &  1  &  1  &  0  &  0  &  0  &  0   \\
0  &  0  &  1  &  0  &  1  &  0  &  0  &  0  &  0  &  1   \\
1  &  0  &  1  &  0  &  0  &  0  &  0  &  1  &  0  &  0   \\
1  &  1  &  0  &  0  &  0  &  0  &  0  &  0  &  0  &  1   \\
0  &  0  &  0  &  1  &  0  &  0  &  0  &  0  &  1  &  1   \\
0  &  0  &  1  &  1  &  0  &  0  &  1  &  0  &  0  &  0   \\
\end{matrix} \right)
\end{equation*}
then $B\in S(10,3)$ and $\det(B)=48$. Next, let
\begin{equation*}
A = \left( \begin{matrix}
0  &  1  &  0  &  0  &  0  &  0  &  0  &  1  &  1  &  0   \\
0  &  0  &  0  &  0  &  1  &  1  &  1  &  0  &  0  &  0   \\
1  &  0  &  0  &  1  &  0  &  0  &  0  &  0  &  1  &  0   \\
0  &  0  &  1  &  0  &  0  &  0  &  0  &  1  &  0  &  1   \\
0  &  1  &  0  &  1  &  0  &  1  &  0  &  0  &  0  &  0   \\
0  &  0  &  0  &  0  &  0  &  1  &  0  &  0  &  1  &  1   \\
1  &  0  &  0  &  0  &  1  &  0  &  0  &  0  &  0  &  1   \\
0  &  0  &  0  &  1  &  0  &  0  &  1  &  1  &  0  &  0   \\
1  &  0  &  1  &  0  &  0  &  0  &  1  &  0  &  0  &  0   \\
0  &  1  &  1  &  0  &  1  &  0  &  0  &  0  &  0  &  0   \\
\end{matrix} \right)
\end{equation*}
then we see $A\in S(10,3)$ and $\det(A)=15<M(10,3)$. Further, we can check that the off-diagonal entries of $A A^T$ and $A^TA$ are exclusively $0$ and $1$ which of course satisfy $|x-2/3|<1$.

\section{Matrices with $kn$ ones}\label{sec:knones}

In~\cite{bruhn}, they show that $M_T(n,2) \le 2^{n/6}3^{n/6}\approx 1.348^n$ giving an exponential improvement over Ryser's theorem. They ask if one can show small bounds for matrices with $3n$ ones. We do this and in fact show that the bound in Theorem~\ref{thm:tworows} holds for matrices in $T(n,k)$ where $k$ is integral.

\begin{theorem}\label{thm:knones}
Let $k\ge 2$ be an integer. Let $c_k =  \left(\sqrt{k^2-1}\right)^{\frac12\left(1 -\frac{1}{k}\right)} k^{\frac{1}{2k}}$ as in Theorem~\ref{thm:tworows}. Then $M_T(n,k) \le c_k^n$
\end{theorem}

\begin{proof}
Let $A\in T(n,k)$. We assume $A$ is non-singular and so the row sums of $A$ are positive integers. Let $r$ be the number of rows not summing to $k$. Let $a_i$ be the sum of the $i$-th row of $A$. If we apply Hadamard's inequality to the rows not summing to $k$ we have
\begin{equation}\label{eqn:knones}
\det(A) \le \left(\prod_{a_i \ne k} \sqrt{a_i} \right) M_R(n-r,n,k).
\end{equation}
We want to show that we can reduce to the case $a_i \in \{k-1,k,k+1\}.$ To begin, suppose that there exist $\epsilon_1,\epsilon_2>1$ such that for some $i,j$, $a_i=k-\epsilon_1$ and $a_j=k+\epsilon_2$. Then if we  replace $a_i$ and $a_j$ with $a_i+1$ and $a_j+1$  then the product in Equation~\ref{eqn:knones} only increases. Iterating this procedure we can assume that for all $i$, $a_i \in \{k-2,k-1,k,k+1,k+2\}$ with at most one of $k-2$ and $k+2$ appearing. Next suppose there is some $a_i=k-2$. Then we do not have $a_j=k+2$ for any $j$ so there must exist $j,\ell$ such that $a_j=a_\ell=k+1$. If we replace $(a_i,a_j,a_k)$ with $(k-1, k-1, k+2)$ then we have increased the product by $(k-1)^2(k+2)-(k-2)(k+1)^2=4$. Iterating this procedure we can assume that $k-2$ does not appear among the $a_i$. So at this point the possible $a_i$ values are $k-1, k, k+1$ and $k+2$. Finally, suppose $k+2$ appears at least twice. Then $k-1$ must appear at least four times, otherwise the average exceeds $k$. So we do the replacement
\begin{equation*}
(k-1,k-1,k-1,k-1,k+2,k+2) \to (k-1,k-1,k-1,k+1,k+1,k+1)
\end{equation*}
which preserves the sum of $6k$ and we see that $(k-1)^2(k+1)^3 - (k-1)^4(k+2)^2 = (3k+5)(k-1)^3>0$. Iterating this procedure we can assume that $k+2$ appears at most once. Let $s=|\{i~:~a_i=k-1\}|$ be the number of times $k-1$ appears. Then the number of $a_i$ greater than $k$ must be one of the following quantities:
\begin{enumerate}[nosep]
\item There are exactly $s$ of the $a_i$ equal to $k+1$.
\item There are exactly $(s-2)$ of the $a_i$ equaling $k+1$ and exactly one equaling $k+2$.
\end{enumerate}
We have the first case if $r$ is even and the second if $r$ is odd. In the first case we have $\prod_{a_i \ne k} a_i = (k-1)^s(k+1)^s$ and in the second case the product is $(k-1)^s (k+1)^{s-2}(k+2)$. The ratio of the first quantity to the second is $(k+1)^2/(k+2)>1$ for all $k\ge 1$. Thus we can conservatively assume we are in the first case. Using Theorem~\ref{thm:tworows} to bound $M_R(n-t,n,k)$ we have
\begin{align*}
\det(A) 	&\le \left(\prod_{a_i \ne k} \sqrt{a_i} \right) M_R(n-t,n,k) \\
		&\le  (k-1)^{s/2}(k+1)^{s/2} M_R(n-2s,n,k) \\
		&\le \sqrt{k^2-1}^{s}  \sqrt{k^2-1}^{\frac{n-2s}{2}-\frac{n}{2k}} k^{\frac{n}{2k}} \\
		&= \sqrt{k^2-1}^{n(1-1/k)}  k^{\frac{n}{2k}} \\
		&= (c_k)^n
\end{align*}
as desired.
\end{proof}

Recalling that $c_{3}=24^{1/6}\approx 1.6984$ we have $M_T(n,3) \le 1.6984^n$. Recall from Section~\ref{sec:k=2} that a construction based on the Fano plane gives the lower bound $M_T(n,3)\ge (24^{1/7})^n\approx 1.5746^n$ for infinitely many $n$. So $\limsup_{n\to\infty} M_T(n,3)^{1/n} \in [24^{1/7},24^{1/6}]$. The authors of~\cite{bruhn} conjecture that $24^{1/7}$ is the true value. We echo this sentiment. At the very least we do not believe our upper bound is tight.

In our proof of Theorem~\ref{thm:knones} we argued that a matrix in $T(n,k)$ that has many rows not summing to $k$ must have determinant smaller than $c_k^n$. If we consider $T(n,\tilde{k})$ for non-integer $\tilde{k}\in(k,k+1)$ it seems reasonable to expect that if the rows of a matrix in $T(n,\tilde{k})$ are not mostly of weight $k$ and $k+1$ in the appropriate ratio then the determinant will be small. As such we have the following conjecture.

\begin{conjecture}\label{conj:ktildeones}
Let $\tilde{k}> 1$ be a real number. Let $k= \lfloor \tilde{k} \rfloor$. Let $\gamma = \tilde{k}- k$. Let $m_1 = (1-\gamma)n$ and $m_2=\gamma n$. Then
\begin{equation*}
M_T(n,\tilde{k}) \le M_R(m_k,n,k)^{1-\gamma}  M_R(m_{k+1},n,k+1)^{\gamma}.
\end{equation*}
\end{conjecture}

Conjecture~\ref{conj:ktildeones} would imply that there exists
\begin{equation*}
d_{\tilde{k}} = \sqrt{k^2-1}^{\frac{1-\gamma}2 - \frac{1}{2k}} \sqrt{k}^{(1-\gamma)/k} \sqrt{(k+1)^2-1}^{\frac{\gamma}2 - \frac{1}{2(k+1)}} \sqrt{k+1}^{\gamma/(k+1)}.
\end{equation*}
such that $\ds T(n,\tilde{k}) \le d_k^n.$ As $d_{\tilde{k}}<\sqrt{k}$ this would show that $M_T(n,\tilde{k})$ is exponentially smaller than $\tilde{k}^{n/2}$ for fixed $\tilde{k}$.

\section{Perturbations}\label{sec:perturbations}

The techniques in this paper can be applied to perturbations of combinatorial matrices. There are many different generalizations one might make. In this section we give a small illustration.

\begin{definition}
 For $\delta\in [0,1)$, let $R_\delta(n,k)$ be the set of $n\times n$ matrices where each row has exactly $k$ non-zero elements each lying in the interval $[1-\delta,1+\delta]$.
\end{definition}
We can think of a matrix in $R_\delta(n,k)$ as a perturbation of a matrix in $R(n,k)$. If $A\in R_\delta(n,k)$ then the rows have norms at most $\sqrt{k}(1+\delta).$ So Hadamard's inequality tells us that $\det(A)\le k^{n/2}(1+\delta)^{n/2}$. The techniques in this paper can be used to improve this bound. We illustrate this with the following generalization of Theorem~\ref{thm:tworows}.

\begin{theorem}\label{thm:perturbed}
If $A\in R_\delta(n,k)$, then $\det(A)\le d_\delta(k)^n$ where
\begin{equation*}
d_\delta(k) = \sqrt{k^2(1+\delta)^2-(1-\delta)^2}^{\frac12(1-1/k)} (k (1+\delta)^2)^{\frac{1}{2k}}. 
\end{equation*}
\end{theorem}
\begin{proof}
 The proof is nearly identical to that of Theorem~\ref{thm:tworows}. If two rows have overlapping nonzero entries their volume is at most 
\begin{equation*}
 \det\left(\begin{matrix} k (1+\delta) & 1-\delta \\ 1-\delta & k(1+\delta) \end{matrix}\right) = \sqrt{k^2(1+\delta)^2-(1-\delta)^2}
\end{equation*}
which is analogous to $\sqrt{k^2-1}$ in the unperturbed case. Once we can no longer guarantee an overlapping pair of rows we apply Hadamard's inequality which uses the max row norm of $k(1+\delta)$.
\end{proof}

If $\delta = o(1/k^2)$ then we will show in Theorem~\ref{thm:perturbed_beat_hadamard} in the appendix that for $k$ sufficiently large, $d_\delta(k) < \sqrt{k}$ so an inequality stronger than Hadamard applied to the unperturbed matrix still holds. One can of course consider perturbations of the zero elements as well. In each of these cases the techniques of Sections~\ref{sec:qrows} and ~\ref{sec:greedy} can be applied.

\subsection*{Example, $k=4$, $\delta=0.01$}

Let $k=4$ and $\delta=0.01$ and suppose $A\in R_\delta(n,k)$. 
\begin{enumerate}[nosep]
\item We have $\sqrt{k}(1+\delta)=2.02$. Thus Hadamard's inequality implies $\det(A)\le 2.02^n$.
\item Using Theorem~\ref{thm:perturbed}, we have $\det(A)\le d_\delta(k)^n\approx 1.9892^n$.
\end{enumerate}

\section{Conclusion and open questions}
We summarize some of our results for various $k$ in Table~\ref{tab:growth_constants}.


\begin{table}[ht]
\centering
\begin{tabular}{|c||c|c|c|c|c|c|}
\hline
$ k  $  &  $ c_{1,k}=\sqrt{k}   $  &  $ c_{2,k}  $  &  $  q_* $  &  $ c_{q_*,k}  $  &  $ \alpha_k  $  &  $ \beta_k  $  \\
\hline
3.0  &  1.7321  &  1.6984  &  2  &  1.6984  &  21.91  &  1.6977  \\
4.0  &  2.0  &  1.9759  &  3 &  1.9719  &  782.53  &  1.9702  \\
5.0  &  2.2361  &  2.2179  &  3  &  2.2116  &  $1.2591\times10^5$  &  2.2097  \\
6.0  &  2.4495  &  2.4352  &  4  &  2.4279  &  $1.0075\times10^8$  &  2.4257  \\
7.0  &  2.6458  &  2.6341  &  4  &  2.6258  &  $4.3557\times10^{11}$  &  2.6240  \\
8.0  &  2.8284  &  2.8187  &  5  &  2.8103  &  $1.0925\times 10^{16}$  &  2.8083  \\
9.0  &  3.0  &  2.9917  &  5  &  2.9828  &  $1.6920 \times 10^{21}$  &  2.9812  \\
10.0  &  3.1623  &  3.1551  &  5  &  3.1462  & $1.7105\times10^{27}$ &  3.1447  \\
\hline
\end{tabular}
\caption{A summary of bounds for $k=3,\ldots,10$, $q_*$ is the optimal value of $q$ that minimizes $c_{q,k}$ for $q=1,\ldots,k$.}
\label{tab:growth_constants}
\end{table}

We have shown that for any $k$ there exists a constant $c(k)<\sqrt{k}$ such that $M_R(n,k)<c(k)^n$ for $n$ sufficiently large. We do not claim that the constants we have found are the best possible. We leave this as an open question. That is, what is $\limsup_{n\to\infty} M_R(n,k)^{1/n}$? For example, is $\limsup_{n\to\infty} M_R(n,3)^{1/n}<\beta_3 = (40/11)^{11/36}2^{7/36} \approx 1.6977$? Recall from Section~\ref{sec:k=2} that we have the lower bound $\limsup_{n\to\infty} M_R(n,3)^{1/n} \ge 24^{1/7} \approx 1.5746$.

We note one avenue through which this work may be improved. For $A\in R(m,n,k)$ let $q_{\text{max}}$ be the maximal column sum of $A$. Then we can take the appropriate  $q_{\text{max}}$ rows and bound their volume. In our approach we use the fact that the matrix resulting after the deletion of these rows lies in $R(m- q_{\text{max}},n,k)$. However, we know the resulting matrix has a zero column since we have removed all ones. Thus we could recursively use an inequality for the volume of $R(m- q_{\text{max}},n-1,k)$. This smaller matrix has a larger density of ones and gives a better bound. This is of course harder to analyze since the maximum column sum depends on $A$.

In Section~\ref{sec:knones} we asked how to extend our bound for $M_R(n,k)$ for integral $k$ to any real value. We gave Conjecture~\ref{conj:ktildeones}. For what values of $n,k$ does $M(n,k)=M_R(n,k)$? We know, for example, that when $\lambda = k(k-1)/(n-1)$ and there is an $(n,k,\lambda)$ combinatorial design this holds. We observed in Section~\ref{sec:k=2} that $M(7,2)\ne M_R(7,2)$. Are there certain values of $k$ for which equality always holds? The same questions apply to $M_R(n,k)$ and $M_T(n,k)$. Finally, we wonder for $\Theta(n^{1/3})\le k < sqrt{n}$, a domain on which no $(n,k,\lambda)$-design exists how much can Ryser's bound be improved?

\section*{Acknowledgments}
The computations that informed this work were performed using a combination of Julia~\cite{julia} and Sage~\cite{sage}. The author is very grateful for such excellent open-source tools. The author also wishes to thank his advisor, Swastik Kopparty, for many helpful discussions and edits.

\appendix 

\section{Appendix}\label{sec:appendix}

\beatryser*

\begin{proof}
We want to show that
\begin{equation}\label{eqn:tworows_ryser_compare}
\left(\left(\sqrt{k^2-1}\right)^{\frac12\left(1 -\frac{1}{k}\right)} k^{\frac{1}{2k}}\right)^n < k (k-\lambda)^{\frac12 (n-1)} =  \frac{k}{k-\lambda}(k-\lambda)^{n/2}.
\end{equation}
Raising both sides to the power $2k/n$ we obtain
\begin{equation*}
\left(\sqrt{k^2-1}\right)^{k-1} k  < \left(\frac{k}{k-\lambda}\right)^{2k/n} (k-\lambda)^k.
\end{equation*}
So it suffices to show
\begin{equation*}
\left(\sqrt{k^2-1}\right)^{k-1} k  < (k-\lambda)^k.
\end{equation*}
Since $\sqrt{k^2-1} < k-\frac{1}{2k}$, it suffices to show
\begin{equation*}
\left(k-\frac{1}{2k}\right)^{k-1} k  < (k-\lambda)^k
\end{equation*}
which simplifies to 
\begin{equation*}
\left(1-\frac{1}{2k^2}\right)^{k-1}   < (1-\lambda/k)^k.
\end{equation*}
Taking logs,
\begin{equation*}
(k-1)\log\left(1-\frac{1}{2k^2}\right)  < k\log (1-\lambda/k),
\end{equation*}
thus
\begin{equation*}
(k-1)\log\left(-\frac{1}{2k^2} + O\left(\frac{1}{k^4}\right)\right) < k \left(\frac{-\lambda}{k} + O\left( (\lambda/k)^2 \right)\right) =-\lambda +O\left(\frac{\lambda^2}k\right).
\end{equation*}
Thus it suffices to show that 
\begin{equation*}
\frac{1}{2k^2} \gg \frac{\lambda}{k-1} = \frac{k}{n-1}
\end{equation*}
which holds provided $k=o(n^{1/3})$.
\end{proof}

Next we show that for large $k$, $c_{q,k}$ is minimized when $q\approx 0.44 k$. 

\begin{theorem}\label{thm:point44}
 Let
\begin{equation*}
c_{q,k} = (q+k-1)^{\frac 1{2q} \left(1-\frac{q-1}k \right)} (k-1)^{\frac 12 \frac {q-1}q \left(1-\frac{q-1}k \right)} k^{\frac{(q-1)}{2k}}
\end{equation*}
as in Theorem~\ref{thm:qrows}. Let 
\begin{equation*}
q_* = \argmin_{q=1,\ldots,k} c_{k,q}.
\end{equation*}
Let $s\approx 0.4395$ be the positive root of
\begin{equation}\label{eqn:point44}
 s^3+s - \log(1+s)(s+1) = 0.
\end{equation}
Then $\ds \lim_{k\to\infty} \frac{q_*}{k} = s$.
\end{theorem}

\begin{proof}
We have
\begin{equation}\label{eqn:c_k_q_squared}
 c_{q,k}^2 = (q+k-1)^{\frac 1{q} \left(1-\frac{q-1}k \right)} (k-1)^{\frac {q-1}q \left(1-\frac{q-1}k \right)} k^{\frac{(q-1)}{k}}.
\end{equation}
Noting that the exponents in equation~\ref{eqn:c_k_q_squared} sum to one we have
\begin{equation*}
 \frac{c_{q,k}^2}{k} = \left(1+\frac{q-1}k\right)^{\frac 1{q} \left(1-\frac{q-1}k \right)} \left(1-\frac1k\right)^{\frac {q-1}q \left(1-\frac{q-1}k \right)}
\end{equation*}
Let $s=(q-1)/k$. Since $c_{2,k}<c_{1,k}$ we can assume $q>1$ and thus $s\in(0,1)$. We have
\begin{equation*}
\frac{c_{q,k}^2}{k} = (1+s)^{\frac{1-s}{sk+1}}\left(1-\frac1k\right)^{\frac{1-s}{sk+1}(1-s)}.
\end{equation*}
Thus
\begin{align*}
G(s,k):=\log\left( \frac{c_{q,k}^2}{k} \right) &= \frac{s-1}{sk+1} \left( \log(1+s) + sk\log\left(1-\frac1k\right)\right) \\
				       &= \frac{s-1}{sk+1} \left(\log(1+s) - s + O\left(\frac1{k^2}\right)\right) \\
				        &= \frac{s-1}{sk+1} \left(\log(1+s) - s \right) + O\left(\frac1{k^3}\right)
\end{align*}
Then,
\begin{equation*}
 \frac{d}{ds} G(s,k) = \frac{ks^3+ks+2s^2 - \log(1+s)(ks+k+s+1)}{(ks+1)^2(s+1)}  + O\left(\frac{1}{k^4}\right)
\end{equation*}
Thus,
\begin{align*}
 \frac{(ks+1)^2(s+1)}{k} \frac{d}{ds} G(s,k) &= (s^3+s - \log(1+s)(s+1)) + O\left(\frac1k\right)
\end{align*}

So the value of $s$ that, asymptotically, minimizes $G(s,k)$ is the positive root of equation~\ref{eqn:point44}.
\end{proof}

Next we show that for constant $k$, Theorem~\ref{thm:greedy_det} gives a better asymptotic than Theorem~\ref{thm:tworows}.

\begin{theorem}\label{thm:greedy_beats_tworows}
Let 
\begin{equation*}
 c_k = \left(\sqrt{k^2-1}\right)^{\frac12\left(1 -\frac{1}{k}\right)} k^{\frac{1}{2k}}
\end{equation*}
as in Theorem~\ref{thm:tworows} and 
\begin{equation*}
 \beta_k = \left(k + \frac{k}{H_{k}}  -1  \right)^{\frac12 (H_k/k)}  (k-1)^{\frac12(1-H_k/k)}
\end{equation*}
as in Theorem~\ref{thm:greedy_det}. Then $\beta_k<c_k$.
\end{theorem}
\begin{proof}
If we raise $\beta_k$ and $c_k$ to the power $2k$ and compare we want to show that
\begin{equation*}
 \left(k + \frac{k}{H_k}-1\right)^{H_k} (k-1)^{k-H_k} < \sqrt{k^2-1}^{k-1} k.
\end{equation*}
Rearranging, this is equivalent to
\begin{equation}\label{eqn:greedy_beats_tworows_compare}
 \left(1 + \frac{1}{H_k}-\frac1k\right)^{H_k} < \frac{\sqrt{k^2-1}^{k-1} }{(k-1)^{k-H_k}}.
\end{equation}
We see that for all $k$ the left hand side of equation~\eqref{eqn:greedy_beats_tworows_compare} is less than $e$. We use the inequality $\sqrt{k^2-1}>k-1/k$ to bound the right hand side.
\begin{align*}
 \frac{\sqrt{k^2-1}^{k-1} }{(k-1)^{k-H_k}} &> \frac{\left(k-\frac1k\right)^{k-1}}{(k-1)^{k-H_k}} \\
&=\left(\frac{k-\frac1k}{k-1}\right)^{k-H_k} \left(k-\frac1k\right)^{H_k-1} \\
&=\left(1+\frac1k\right)^{k-H_k}  \left(k-\frac1k\right)^{H_k-1} \\
&> 1 \cdot k = k
\end{align*}
for $k\ge 4$. Since $4>e$ the result holds for $k\ge 4$ and one easily check that it holds for $k<4$.
\end{proof}

\begin{theorem}\label{thm:perturbed_beat_hadamard}
 Let $d_\delta(k) = \sqrt{k^2(1+\delta)^2-(1-\delta)^2}^{\frac12(1-1/k)} (k (1+\delta)^2)^{\frac{1}{2k}}$ as in Theorem~\ref{thm:perturbed}. Then for $\delta = o(1/k^2)$, $d_\delta(k)<\sqrt{k}$.
\end{theorem}

\begin{proof}
 Raising both sides of the inequality $d_\delta(k)<\sqrt{k}$ to the power $2k$ we find
\begin{equation*}
 \sqrt{k^2(1+\delta)^2-(1-\delta)^2}^{k-1} k (1+\delta)^2 < k^k
\end{equation*}
which we can simplify to 
\begin{equation}\label{eqn:perturbed_compare}
 \left(\frac{\sqrt{k^2(1+\delta)^2-(1-\delta)^2}}{k}\right)^{k-1} < \frac{1}{(1+\delta)^2}.
\end{equation}
We simplify and apply the inequality $\sqrt{a^2-b^2}<a-\frac{b^2}{2a}$ in the left hand side of equation~\eqref{eqn:perturbed_compare} to obtain
\begin{align*}
 \left(\frac{\sqrt{k^2(1+\delta)^2-(1-\delta)^2}}{k}\right)^{k-1} &= \sqrt{(1+\delta)^2 - (1-\delta)^2/k^2}^{k-1} \\
&\le \left(1+\delta - \frac{(1-\delta)^2}{2k^2(1+\delta)}\right)^{k-1} \\
&=\left(1-\frac{1}{2k^2} + o\left(\frac{1}{k^2}\right)\right)^{k-1} \\
&= 1-\frac{1}{2k} + o\left(\frac{1}{k^2}\right).
\end{align*}
The right hand side of equation~\eqref{eqn:perturbed_compare} is
\begin{equation*}
 \frac{1}{(1+\delta)^2} = 1-\delta^2 + o(\delta^2) = 1-o\left(\frac{1}{k^4}\right).
\end{equation*}
So we see the inequality holds.
\end{proof}


\bibliography{maximal_determinant_k_ones.bib}
\bibliographystyle{plain}

\end{document}